\documentclass[10pt]{amsart}

 \usepackage{amsmath,amsthm}
\usepackage{amssymb,latexsym}

\makeatletter
\@addtoreset{equation}{section}
\makeatother

\newtheorem{theorem}{Theorem}[section]
\newtheorem{proposition}[theorem]{Proposition}
\newtheorem{lemma}[theorem]{Lemma}
\newtheorem{corollary}[theorem]{Corollary}

\theoremstyle{definition}

\newtheorem{remark}[theorem]{Remark}

\newcommand{\wbar}[1]{\overline{#1}}
\newcommand{\what}[1]{\widehat{#1}}
\newcommand{\wtil}[1]{\widetilde{#1}}

\newcommand{\id}{\operatorname{id}}

\newcommand{\supp}{\operatorname{supp}}

\newcommand{\fA}{\mathcal{A}}
\newcommand{\fB}{\mathcal{B}}
\newcommand{\fC}{\mathcal{C}}

\newcommand{\fH}{\mathcal{H}}
\newcommand{\fK}{\mathcal{K}}

\newcommand{\fW}{\mathcal{W}}
\newcommand{\fV}{\mathcal{V}}

\newcommand{\Ree}{\mathbb{R}}

\newcommand{\En}{\mathbb{N}}

\newcommand{\alp}{\alpha}

\newcommand{\Del}{\Delta}

\newcommand{\lam}{\lambda}

\newcommand{\sig}{\sigma}

\newcommand{\cb}{\mathrm{cb}}
\newcommand{\bl}{\mathrm{L}}
\newcommand{\fal}{\mathrm{A}}

\newcommand{\meas}{\mathrm{M}}

\newcommand{\vn}{\mathrm{VN}}

\newcommand{\ideal}{\mathrm{I}}

\newcommand{\ball}{\mathrm{b}}
\newcommand{\mat}{\mathrm{M}}

\begin{document}

\title[Corrigendum]
{Corrigendum:  Similarity degree of Fourier algebras}

\author{Hun Hee Lee, Ebrahim Samei and Nico Spronk}

\maketitle

\begin{abstract}
We address two errors made in our paper \cite{leess}.  The most significant error is in Theorem 1.1.
We repair this error, and show that the main result, Theroem 2.5 of \cite{leess}, is true.
The second error is in one of our examples, Remark 2.4 (iv), and we partially resolve it.
\end{abstract}

\footnote{{\it Date}: \today.


The first named author would like to thank Seoul National University through the Research Resettlement Fund for the new faculty and the Basic Science Research Program through the National Research Foundation of Korea (NRF), grant  NRF-2015R1A2A2A01006882.
The second named author would like to thank NSERC Grant 409364-2015.
The third named author would like thank NSERC Grant 312515-2015.}

\section{On Theorem 1.1 of \cite{leess}}  

\subsection{A corrected version of the threorem}
We begin with a simple observation whose straightforward proof we omit.

\begin{lemma}\label{lem:embedding}
Let $\fV$ be a an operator space and $\fV_0$ and $\fW$ be subspaces with $\fV_0$ dense in $\fV$,
$\fV_0\cap\fW$ dense in $\fW$, and $\fW$ closed in $\fV$.  Then the map
\[
v+\fV_0\cap\fW\mapsto v+\fW:\fV_0/(\fV_0\cap\fW)\to\fV/\fW
\]
is a complete isometry.  Hence we identify $\fV_0/(\fV_0\cap\fW)$ as a subspace of $\fV/\fW$ and, for any $n \in \En$,
the quotient map takes the matricial open unit ball $\ball_1(\mat_n\otimes\fV_0)$ onto 
$\ball_1(\mat_n\otimes[\fV_0/(\fV_0\cap\fW)])\cong\ball_1([\mat_n\otimes \fV_0]/[\mat_n\otimes (\fV_0\cap\fW)])$.
\end{lemma}

We now recall the notation of \cite{leess}.  For a Banach algebra and 
operator space $\fA$, $c\geq 1$, $\wtil{\fA}_c$ (contained in $\fB(\fH_c)$)
is the universal operator algebra generated by representations  
on Hilbert spaces $\pi:\fA\to\fB(\fH)$ with completely bounded norm $\|\pi\|_{\cb}\leq c$, and $\iota_c:\fA\to\wtil{\fA}_c$ is
the canonical embedding.  Note that {\it we assume that $\iota_1$ is injective.}  
We say that $\fA$ satisfies the {\it similarity property} for completely bounded homomorphims
if for each completely bounded  homomorphim $\pi:\fA\to\fB(\fH)$, there is an invertible $S$ in $\fB(\fH)$
for which $\|S\pi(\cdot)S^{-1}\|_{\cb}\leq 1$.
We also consider the ``weighted multiplication" 
map on the $N$-fold Haagerup tensor product of $\fA$ with itself, $m_{N,c}:\fA^{N\otimes^h}\to\wtil{\fA}_c$, given 
on elementary tensor by
\[
m_{N,c}(u_1\otimes\dots\otimes u_N)
=\frac{1}{c^N}\iota_c(u_1)\dots\iota_c(u_n)=\frac{1}{c^N}\iota_c(u_1,\dots u_N)
\]
which is a complete contraction.

We let $m_N:\fA^{\otimes N}\to\fA$ be the multiplication map on the $N$-fold algebraic tensor product and
$\fA^N=m_N(\fA^{\otimes N})$.  We say that $\fA$ is {\it square-dense} provided
the closure, $\wbar{\fA^2}$, is all of $\fA$.  This clearly implies that $\wbar{\fA^N}=\fA$ for all $N\geq 2$.

\begin{theorem} \label{theo:mainrepair}
Suppose that 
$m_{N,1}:\fA^{N\otimes^h}\to\wtil{\fA}_1$ is a complete
surjection, i.e.\ the induced map $\mu_{N,1}:\fA^{N\otimes^h}/\ker m_{N,1}\to\wtil{\fA}_1$ is a complete 
isomorphism which is
completely bounded below by $1/K$, and either of the following conditions holds:

{\bf (i)} $\fA$ is square-dense and $\ker m_N$ is dense in  $\ker m_{N,1}$; or

{\bf (ii)} $\fA$ satisfies the similarity property for completely bounded homomorphisms.

\noindent Then any completely bounded homomorphism
$\pi:\fA\to\fB(\fH)$ admits an invertible operator $S$ in $\fB(\fH)$ for which
\begin{equation}\label{eq:simdeg}
\|S\pi(\cdot)S^{-1}\|_{\cb}\leq 1\text{ and }\|S\|\|S^{-1}\|\leq K\|\pi\|_{\cb}^N.
\end{equation}
\end{theorem}

In (i) we add to \cite[Theorem 1.1]{leess} both the assumption of square-density of $\fA$, and density
of $\ker m_N$ in $\ker m_{N,1}$, and gain the similarity property.  In (ii) we assume the similarity
property, but gain information about {\it completely bounded similarity degree} $d_\cb(\fA)$: 
the smallest $N$ for which (\ref{eq:simdeg}) holds.

In the proof of  \cite[Theorem 1.1]{leess}, it is indicated that $\iota_c\circ\iota_1^{-1}:\iota_1(\fA)\to\wtil{\fA}_c$
extends to a completely bounded map on $\wtil{\fA}_1$.  There is a gap in that proof, which we repair, below.

\begin{proof}[Proof of Theorem \ref{theo:mainrepair}]
We begin with assumptions of (i).
Let $\fK_N=\ker m_N$ and $\fK_{N,1}=\ker m_{N,1}$.
The injectivity of $\iota_1$ provides that
\begin{equation*}%\label{eq:kernel}
\fK_N=\fA^{N\otimes}\cap\fK_{N,1}.
\end{equation*}
In particular, we may use Lemma \ref{lem:embedding} to regard $\fA^{N\otimes}/\fK_N$ as a dense subspace
of $\fA^{N\otimes^h}/\fK_{N,1}$ with
\begin{equation}\label{eq:murange}
\mu_{N,1}(\fA^{N\otimes}/\fK_N)=m_{N,1}(\fA^{N\otimes})=\iota_1\circ m_N(\fA^{\otimes N})=\iota_1(\fA^N).
\end{equation}
Let $n\in\En$ and $a$ be in the open unit ball 
$\ball_1(\mat_n\otimes \iota_1(\fA^N))\subseteq \ball_1(\mat_n\otimes \wtil{\fA}_1)$.
Our assumptions on $\mu_{N,1}$ provide
a  $T$ in the open $K$-ball $\ball_K(\mat_n\otimes (\fA^{N\otimes^h}/\fK_{N,1}))$,
for which $\id_n\otimes\mu_{N,1}(T)=a$.  Moreover, since $\mu_{N,1}$ is a bijection, 
(\ref{eq:murange}) shows that $T\in\mat_n\otimes(\fA^{N\otimes}/\fK_N)$,
whence Lemma \ref{lem:embedding} and homogeneity provide
a $t \in \ball_K(\mat_n\otimes\fA^{N\otimes})\subseteq \ball_K(\mat_n\otimes\fA^{N\otimes^h})$
for which $T=t+\mat_n\otimes\fK_N$.  

Now we fix $c\geq 1$.  We observe that
\begin{equation}\label{eq:mueq}
\iota_c\circ\iota_1^{-1}\circ m_{N,1}|_{\fA^{N\otimes}}=c^N m_{N,c}|_{\fA^{N\otimes}}
:\fA^{N\otimes}\to\wtil{\fA}_c.
\end{equation}
If $a\in\ball_1(\mat_n\otimes \iota_1(\fA^N))$, choose $T$ and $t$ as above, so
$a=\mu_{N,1}(T)=m_{N,1}(t)$ and we use (\ref{eq:mueq}) to see that
\[
\|\id_n\otimes \iota_c\circ\iota_1^{-1}(a)\|=\|\id_n\otimes \iota_c\circ\iota_1^{-1}\circ m_{N,1}(t)\|
=c^N\|\id_n\otimes m_{N,c}(t)\|\leq Kc^N.
\]
Taking supremum over all such $a$ and all such $n$ yields that 
$\| \iota_c\circ\iota_1^{-1}\|_\cb\leq Kc^N$ on $\iota_1(\fA^N)$.
The assumption of square density of $\fA$ provides that $\fA^N$ in dense in $\fA$, 
whence $\iota_1(\fA^N)$ is dense in $\wtil{\fA}_1$, and thus we find that
$\iota_c\circ\iota_1^{-1}$ extends to a map $\iota_{1,c}:\wtil{\fA}_1\to\wtil{\fA}_c$ with $\|\iota_{1,c}\|_\cb\leq Kc^N$,
as desired.

If we assume (ii), then $\iota_c:\fA\to\wtil{\fA}_c\subseteq\fB(\fH_c)$ is similar to a complete contraction and hence
$\iota_c\circ\iota_1^{-1}$ is automatically completely bounded.  Then we can simplify (\ref{eq:mueq}) above
to $\iota_c\circ\iota_1^{-1}\circ m_{N,1}=c^N m_{N,c}$ on $\fA^{N\otimes^h}$, and gain the same estimate
on $\|\iota_{1,c}\|_\cb$, as above, using only the complete surjectivity of $m_{N,1}$.

The rest of proof follows as in \cite[Theorem 1.1]{leess}.  \end{proof}

\begin{remark}
 In the case where $\fA$ is unital, and hence square-dense, Theorem \ref{theo:mainrepair} with assumptions (i)
is a partial converse to \cite[Theorem 2.5]{pisier}, where Pisier shows that {\it if $\fA$ satisfies
a certain similarity property, then $m_{N,1}:\fA^{N\otimes^h}\to\wtil{\fA}_1$ is a complete surjection}.
Our result with assumptions (ii) is really the aspect of \cite[Theorem 2.5]{pisier} which
begins from the complete surjectivity result to obtain (\ref{eq:simdeg}).  We include it merely for context
and completeness of presentation.

Regrettably, the essential error of the authors in \cite[Theorem 1.1]{leess} is also made in  
\cite[Theorem 4.2.9]{spronkT}.  \end{remark}

\subsection{On applying the corrected theorem to Fourier algebras}
Let $G$ be a locally compact group and $\fal(G)$ its Fourier algebra.  
We used the flawed \cite[Theorem 1.1]{leess} to prove our main theorem \cite[Theorem 2.5]{leess}.
We wish to deduce the latter result from our present Theorem \ref{theo:mainrepair}, instead.

In the proof of \cite[Theorem 2.5]{leess} we successfully
showed, for a quasi-small invariant 
neighbourhood (QSIN) group $G$, that
\begin{equation}\label{eq:cquot}
m=m_{2,1}:\fal(G)\otimes^h\fal(G)\to\fC_0(G)\text{ is a complete quotient map.}
\end{equation}
This was achieved by showing that
\begin{equation}\label{eq:madjisom}
m^*:\meas(G)\to\vn(G)\otimes^{eh}\vn(G)\text{ is a complete isometry}
\end{equation}
It is shown in \cite[Proposition 2.1]{leess}, that completely contractive homomorphisms
are exactly the $*$-homomorphisms for Fourier algebras of QSIN groups, and hence
$\fC_0(G)=\wtil{\fal}(G)_1$.  This gives the first
assumption of our present Theorem \ref{theo:mainrepair}.  In order to apply this theorem to
prove \cite[Theorem 2.5]{leess}, we shall verify the two aspects of
condition (i) of Theorem \ref{theo:mainrepair}.
  To see the first aspect we have that 
\begin{equation}\label{eq:sqdense}
\fal(G)\text{ is square-dense} 
\end{equation}
thanks to the Tauberian theorem of \cite{eymard}.  The second aspect is more delicate.

It is observed in \cite[Section 3]{alaghmandantt} that $\fal(G)\otimes^h\fal(G)$ is a semisimple commutative
Banach algebra, and that $\fal(G\times G)\cong\fal(G)\what{\otimes}\fal(G)$ (operator projective tensor
product) completely contractively embeds in $\fal(G)\otimes^h\fal(G)$ with dense range.  Hence this is a regular
Banach algebra on its Gelfand spectrum $G\times G$.  Given a closed subset $E$ of $G\times G$ we let
\begin{align*}
\ideal_h(E)&=\{u\in\fal(G)\otimes^h\fal(G):u|_E=0\}\text{, and} \\
\ideal^0_h(E) &=\{u\in\ideal_h(G):\supp(u)\cap E=\varnothing\text{ and }\supp u\text{ is compact}\}.
\end{align*}
We say that $E$ is a {\it set of spectral synthesis} if $\wbar{\ideal^0_h(E)}=\ideal_h(E)$.
In the regular algebra $\fal(G)\otimes^h\fal(G)$, this implies that any ideal with vanishing set $E$
is dense in $\ideal_h(E)$.  See, for example, the recent book \cite{kaniuth}.

The following is shown as \cite[Corollary 4.3]{alaghmandancn}, where it was proved in response
to a question to M. Alaghmandan by the authors.  
This result was previously shown for compact groups  in \cite[Proposition 3.1]{rostamis},
and virtually abelian groups in \cite[Corollary 5.4]{alaghmandantt}.
For convenience of the reader, we provide a proof of
this general result, based on (\ref{eq:madjisom}).

\begin{proposition}\label{prop:delspsyn}
Let $G$ be a QSIN group.  Then the diagonal $\Del=\{(s,s):s\in G\}$ is a set of spectral sunthesis
for $\fal(G)\otimes^h\fal(G)$.
\end{proposition}

\begin{proof}
The contraction $\fal(G\times G)\hookrightarrow\fal(G)\otimes^h\fal(G)$ gives contractive
injection $\vn(G)\otimes^{eh}\vn(G)\hookrightarrow\vn(G\times G)$.  We let
$\lam_\Del:G\to\vn(G\times G)$ be given by $\lam_\Del(s)=\lam(s,s)$ (left regular representaion)
so $\vn_\Del(G\times G)=\lam_\Del(G)''$ is isomorhic to $\vn(G)$.  Notice that in
$\vn(G)\otimes^{eh}\vn(G)\subseteq\vn(G\times G)$, we have for $\mu$ in $\meas(G)$ that
\begin{equation}\label{eq:madjislam}
m^*(\mu)=\int_G\lam(s)\otimes\lam(s)\,d\mu(s)=\int_G\lam(s,s)\,d\mu(s)=\lam_\Del(\mu)
\end{equation}
where each integral is understood as a weak* integral with respect to the respective predual.

Now let $\ideal(\Del)$ and $\ideal^0(\Del)$ be the ideals in $\fal(G\times G)$, defined analagously
to $\ideal_h(\Del)$ and $\ideal^0_h(\Del)$, above.  Since $\Del$ is a closed subgroup,
\cite{takesakit} provides that $\Del$ is a set of spectral synthesis for $\fal(G\times G)$, and
further, that $\ideal(\Del)^\perp=\vn_\Del(G\times G)$.  Hence in 
$\vn(G)\otimes^{eh}\vn(G)\subseteq\vn(G\times G)$ we have that
\[
\ideal^0_h(\Del)^\perp\subseteq\ideal^0(\Del)^\perp=\ideal(\Del)^\perp=\vn_\Del(G\times G).
\]
Let $T\in\ideal^0_h(\Del)^\perp$.  Then by decomposing into a sum of two self-adjoint operators,
the Kaplansky density theorem provides a net $(\mu_\alp)$ from the space of measures
$\meas(G)$ for which
\[
T=\text{w*-}\lim_\alp\lam_\Del(\mu_\alp)\text{, with each }
\|\lam_\Del(\mu_\alp)\|\leq 2\|T\|\text{ in }\vn_\Del(G\times G).
\]
Combining with (\ref{eq:madjisom}) and (\ref{eq:madjislam}) we thus see that each
\[
\|\mu_\alp\|_{\meas(G)}=\|m^*(\mu_\alp)\|_{\vn(G)\otimes^{eh}\vn(G)}
\leq\|\lam_\Del(\mu_\alp)\|\leq 2\|T\|.
\]
Thus, by dropping to subnet, we may suppose that $\mu=\text{w*-}\lim_\alp\mu_\alp$
exists in $\meas(G)$.  But then $T=\lam_\Del(\mu)$ in $\vn_\Del(G\times G)$, hence,
using again (\ref{eq:madjislam}), $T=m^*(\mu)$ in $\vn(G)\otimes^{eh}\vn(G)$.
Thus, if $u\in\ideal_h(\Del)$, then $\langle T,u\rangle=\int_Gu(s,s)\,d\mu(s)=0$, so
$T\in\ideal_h(\Del)^\perp$.  

In summary, we have shown that $\ideal^0_h(\Del)^\perp\subseteq\ideal_h(\Del)^\perp$, so
the bipolar theorem shows that $\ideal_h(E)\subseteq\wbar{\ideal^0_h(E)}$, which
establishes the equality of these ideals.
\end{proof}

\begin{corollary}\label{cor:denseker}
The multiplication map $m_2:\fal(G)\otimes\fal(G)\to\fal(G)$ satisfies that
$\wbar{\ker m_2}=\ker m_{2,1}$ in $\fal(G)\otimes^h\fal(G)$.
\end{corollary}

\begin{proof}
Since $\fK_2=\ker m_2$ is an ideal in $\fal(G)\otimes\fal(G)$, its closure
$\wbar{\fK_2}$ is an ideal in $\fal(G)\otimes^h\fal(G)$.  The regularity of
$\fal(G)$ provides that the vanishing set of $\fK_2$, hence that of $\wbar{\fK_2}$, is $\Del$.
Thus $\wbar{\fK_2}=\ideal_h(\Del)$, by Proposition \ref{prop:delspsyn}.
But $\ideal_h(\Del)=\ker m_{2,1}$.
\end{proof}

The combination of (\ref{eq:cquot}), (\ref{eq:sqdense}) and Corollary \ref{cor:denseker}
give the assumptions Theorem \ref{theo:mainrepair} with condition (i) for $\fal(G)$ with
$G$ a QSIN group.
We hence conclude that \cite[Theorem 2.5]{leess} is true:  {\it if $G$ is a QSIN group, then
any completely bounded homomorphism $\pi:\fal(G)\to\fB(\fH)$ admits an
invertible $S$ in $\fB(\fH)$ for which
\[
S\pi(\cdot)S^{-1}\text{ is a $*$-homomorphism, and }\|S\|\|S^{-1}\|\leq \|\pi\|_{\cb}^2.
\]}

The proofs of results \cite[Corollaries 2.9 \& 2.10]{leess} relied only on \cite[Theorem 2.5]{leess}, as stated, and hence 
remain valid.

\section{On Remark 2.4 (iv) of \cite{leess}}  

This remark is in error, as stated.  The failure of the a
$\Gamma$-space to be amenable, in a sense to
which the authors implicitly appeal, does not imply non-existence of invariant means.

We can partially recover this result, with a simple adaptation of an argument in the preprint \cite{forrestsw}.
We shall use terminology as introduced in \cite{leess}.

Let $\Gamma$ be a dense subgroup of $S=\mathrm{SL}_2(\Ree)$, treated as a discrete group.
Then $G=\Ree^2\rtimes\Gamma$ is not QSIN.  Indeed, as noted in the proof of \cite[Theorem 2.5]{leess},
the QSIN condition would provide an asymptotically inner invariant net $(v_\alp)\subset \bl^1(G)$ which would 
further satisfy
\[
v_\alp\geq 0,\;\int_G v_\alp=1\text{ and }\supp(v_\alp)\searrow\{e\}.
\]
Hence we may suppose that $(v_\alp)$ is supported in the open subgroup
$\Ree^2$, and can be realized as a net in $\bl^1(\Ree^2)$ which satisfies
\[
\|v_\alp\circ\sig-v_\alp\|_1\to 0\text{ for }\sig\text{ in }\Gamma.
\]
We note that as $\Ree^2\setminus\{0\}\cong S/P$ where $P$ is the amenable fixed point subgroup
of any non-zero element of $\Ree^2$.  Hence we regard $(v_\alp)$
as a net of means on the left uniformly continuous functions
$\mathcal{LUC}(S/P)$; any cluster point of this net gives a $\Gamma$-invariant, hence
$S$-invariant mean.  But this violates \cite[\S 3, 1$^\circ$]{eymardM}.

\subsection{Ackowledgments}
The authors are grateful to S.-G. Youn, whose work helped us to uncover the error with \cite[Theorem 1.1]{leess}.
J. Crann kindly pointed out to us the error in \cite[Remark 2.4 (iv)]{leess}.  This work was partially prepared while
the first and third named authors were visiting China.  N.S. is grateful to H. Li, of Chongqing University, and both H.H.L. and N.S. are grateful to Q. Xu, of Harbin Institute of Technology, for their generous hosting, in July of 2018.

%\vfill
%\pagebreak

Addresses:

\noindent {\sc 
Department of Mathematical Sciences  and Research Institute of Mathematics, Seoul National University,
Gwanak-ro 1, Gwanak-gu, Seoul 08826, Republic of Korea \\ \\
Department of Mathematics and Statistics, University of Saskatchewan,
Room 142 McLean Hall, 106 Wiggins Road
Saskatoon, SK, S7N 5E6  \\ \\
Department of Pure Mathematics, University of Waterloo,
Waterloo, ON, N2L 3G1, Canada.}

\medskip
Email-adresses:
\linebreak
{\tt hunheelee@snu.ac.kr}, {\tt samei@math.usask.ca}, {\tt nspronk@uwaterloo.ca}

\end{document}